\documentclass[10pts,twocolumn]{article}

\usepackage{graphicx}      % include this line if your document contains figures
\usepackage{natbib}        % required for bibliography
\usepackage{amsmath}
\usepackage{amssymb}
\usepackage{bm}
\usepackage{diagbox}
\usepackage{nicefrac}
\usepackage{placeins}
\usepackage{svg}
\usepackage{xcolor}

\usepackage{authblk}

\renewcommand{\atop}[2]{\genfrac{}{}{0pt}{}{#1}{#2}}

\newtheorem{theorem}{Theorem}
\newtheorem{proposition}[theorem]{Proposition}
\newtheorem{lemma}[theorem]{Lemma}
\newtheorem{remark}{Remark}
\newenvironment{proof}{%
\textit{Proof.}
}{%
\, \hfill $\square$
}

\begin{document}
%\begin{frontmatter}

\title{Discretization of the Burgers' equation as a port-Hamiltonian system}%\thanksref{footnoteinfo}} 

%\thanks[footnoteinfo]{Sponsor and financial support acknowledgment
%goes here. Paper titles should be written in uppercase and lowercase
%letters, not all uppercase.}

\author[1]{Lorenzo Agostini} 
\author[1]{Michel Fournié} 
\author[1]{Ghislain Haine}

\affil[1]{Fédération ENAC ISAE-SUPAERO ONERA, Université de Toulouse, Toulouse, France (e-mail: ghislain.haine@isae-supaero.fr).}
%\address[Second]{Colorado State University, 
%   Fort Collins, CO 80523 USA (e-mail: author@lamar. colostate.edu)}
%\address[Third]{Electrical Engineering Department, 
%   Seoul National University, Seoul, Korea, (e-mail: author@snu.ac.kr)}

\date{}

\maketitle

\begin{abstract} % Abstract of 50--100 words
The numerical simulation of the inviscid Burgers' equation is often hindered by spurious oscillations near discontinuities. To mitigate this issue, a viscous term can be introduced, leading to the viscous Burgers' equation. In this work, port-Hamiltonian formulations for both the inviscid and the viscous Burgers' equations are proposed, enabling a representation that incorporates both convective and dissipative effects. Boundary control and observation are naturally handled within this framework. Applying a dedicated finite element method, a finite-dimensional port-Hamiltonian system is derived. The relationship between time step, spatial resolution, and viscosity required to achieve numerical stability is analyzed. Numerical experiments validate the effectiveness of the approach.

\end{abstract}

%\end{frontmatter}
%===============================================================================

\vspace{-0.5em}
\section{Introduction}
\label{sec:introduction}
\vspace{-0.5em}

Port-Hamiltonian (pH) systems generalize classical Hamiltonian mechanics by introducing ports that account for energy exchanges with the environment and for energy loss due to dissipation. This structural extension guarantees passivity by construction. Moreover, pH systems are closed under interconnections that preserve power, making them particularly effective for the modular modeling of complex physical systems. From a control perspective, the intrinsic passivity of pH formulations supports the design of stable controllers using passivity-based techniques. A detailed introduction to the theory of pH systems can be found in~\cite{van_der_Schaft_2014}. A comprehensive review of developments in distributed pH systems over the last two decades is provided in~\cite{Rashad_2020}.

In this work, the pH formulation is applied to the Burgers' equation (BE), a fundamental non-linear partial differential equation. The BE is one of the seminal equations in fluid dynamics. It serves as a one-dimensional analogue of the Navier–Stokes equations in the pressureless limit and models non-linear wave propagation in incompressible media. Notably, the BE plays a central role in understanding the formation of shock waves, making it a widely adopted model far beyond its original hydrodynamic context.

The BE was introduced in 1939 (\cite{Burgers_1995}) who simplified the full Navier–Stokes equations by neglecting the pressure term. A comprehensive overview of the historical development and recent advances related to the BE can be found in~\cite{Enflo_2004,Bonkile_2018}.
While the inviscid Burgers' equation (IBE) captures key features of shock formation, its simulation may lead to nonphysical oscillations. To address this issue, a small viscosity term is introduced, motivated by classical numerical schemes commonly used, which tend to introduce artificial viscosity. This modification yields a model that retains the essential characteristics of the original system while being more amenable to stable numerical approximation. The pH formulation is then applied to the viscous Burgers' equation (VBE), which combines convective transport with diffusive effects. In this framework, viscous dissipation is modeled via a resistive port. Boundary ports are naturally incorporated, enabling the definition of boundary control and observation.

The paper is organized as follows. Section~\ref{sec:inviscid-model} introduces a pH formulation of the IBE, including the associated Dirac structure, weak formulation, and structure-preserving finite element discretization. The resulting discrete power balance is established, and numerical simulations highlight the limitations of the inviscid formulation when steep gradients develop. These observations motivate the introduction of viscosity, which is addressed in Section~\ref{sec:viscous-model}. There, the VBE is formulated within the pH framework by adding a distributed resistive port and the corresponding boundary ports, and a structure-preserving discretization is derived. Section~\ref{sec:simulations} presents numerical simulations illustrating the behavior of both inviscid and viscous models and validating the proposed formulation. Section~\ref{sec:stability} provides a parametric numerical study that investigates the influence of the spatial resolution, time-step scaling, and viscosity on the energetic stability of the scheme. Finally, Section~\ref{sec:conclusion} concludes the paper and outlines future works.

\vspace{-0.5em}
\section{Inviscid model}
\label{sec:inviscid-model}
\vspace{-0.5em}

\vspace{-0.5em}
\subsection{Power balance}
\vspace{-0.5em}

Consider the inviscid Burgers' equation (IBE)
$$
\partial_t v(t,x) + \partial_x \frac{v^2(t,x)}{2} = 0, \quad \forall t \ge 0, x \in (0,1),
$$
together with boundary conditions, where $v : \mathbb{R}^+ \times (0,1) \to \mathbb{R}$ is the velocity.

At the formation of a shock located at $x = x_s(t)$, we define the left limits
\begin{center}
$
v_l := \lim_{h \rightarrow 0^+} v(t, x_s(t) - h) \;, \; x_s^-(t) := \lim_{h \rightarrow 0^+} x_s(t) - h \;,
$
\end{center}
and the right limits
\begin{center}
$
v_r := \lim_{h \rightarrow 0^+} v(t, x_s(t) + h) \;, \; x_s^+(t) := \lim_{h \rightarrow 0^+} x_s(t) + h \;.
$
\end{center}
Furthermore, we recall that at a shock, $v_l>v_r$.
\begin{lemma}
\label{lem:kinetic-energy}
For smooth solutions of the IBE, the kinetic energy
$
E(t) := \int_0^1 \frac{v^2(t,x)}{2} \, \mathrm{d}x\;,
$
satisfies the power balance
{\small
$$
\frac{\mathrm{d}E}{\mathrm{d}t} = \frac{v^3(t,0) - v^3(t,1)}{3}\;.
$$}
After the formation of a shock, the time derivative of the kinetic energy becomes
{\small
\begin{equation}
\label{eq:power-balance}
\frac{\mathrm{d}E}{\mathrm{d}t} = \frac{v^3(t,0) - v^3(t,1)}{3} + \frac{(v_r - v_l)^3}{12}\;.
\end{equation}}
In other words, a dissipation occurs at shock (as $v_l>v_r$).
\end{lemma}

\begin{proof}
The proof relies on the key fact that since $v$ strongly satisfies the IBE on both sides of the shock, it follows that $\partial_t (\nicefrac{v^2}{2}) = v \partial_t v = - v \partial_x (\nicefrac{v^2}{2}) = - v^2 \partial_x v = - \partial_x (\nicefrac{v^3}{3})$.

Splitting the integral at the shock yields
{\small
$$
E(t) = \int_0^{x_s(t)} \frac{v^2}{2} \, \mathrm{d} x + \int_{x_s(t)}^1 \frac{v^2}{2} \, \mathrm{d} x\;.
$$}
By applying Leibniz's rule, it follows that
\begin{center}
$
\frac{\mathrm{d} E}{\mathrm{d} t} = \int_0^{x_s(t)} \partial_t \left( \frac{v^2}{2} \right) \mathrm{d} x + \int_{x_s(t)}^1 \partial_t \left( \frac{v^2}{2} \right) \mathrm{d} x + \frac{v_l^2-v_r^2}{2} x_s'(t)\;.
$
\end{center}
Substituting $\partial_t (v^2/2)$ by $-\partial_x (v^3/3)$, one obtains
\begin{center}
$
\frac{\mathrm{d} E}{\mathrm{d} t} = -\int_0^{x_s(t)} \partial_x \left( \frac{v^3}{3} \right) \mathrm{d} x - \int_{x_s(t)}^1 \partial_x \left( \frac{v^3}{3} \right) \mathrm{d} x + \frac{v_l^2-v_r^2}{2} x_s'(t)\;.
$
\end{center}
Evaluating the integrals yields
{\small
\begin{align*}
\frac{\mathrm{d} E}{\mathrm{d} t} &= -\left[ \frac{v^3}{3} \right]_0^{x_s^-} -\left[ \frac{v^3}{3} \right]_{x_s^+}^1 + \frac{v_l^2-v_r^2}{2} x_s'(t) \\
&= \frac{v^3(t,0)}{3} - \frac{v_l^3}{3} + \frac{v_r^3}{3} - \frac{v^3(t,1)}{3} + \frac{v_l^2-v_r^2}{2} x_s'(t)\;.
\end{align*}}
By Rankine–Hugoniot condition~\cite[p.~31]{LeVeque_1992}, the shock speed satisfies
$
x_s'(t) = \frac{\nicefrac{v_r^2}{2} - \nicefrac{v_l^2}{2}}{v_r - v_l} = \frac{v_r + v_l}{2}\;.
$
Substituting yields the result.
\end{proof}

To ensure a physically consistent behavior, a numerical scheme should be able to:
\begin{itemize}
\item capture a shock when it appears;
\item lead to a coherent discrete power balance.
\end{itemize}
The pH formalism, together with an appropriate application of the finite element method, answers the latter. The former is the purpose of the present work.

\vspace{-0.5em}
\subsection{A change of Hamiltonian}
\vspace{-0.5em}

In the pH framework, a Hamiltonian $\mathcal{H}$ has to be chosen, and different Hamiltonians lead to different representations. Often, a physical quantity, such as energy or entropy, is a natural candidate, and the power balance is encoded in a so-called Dirac structure. We refer the reader to~\cite{van_der_Schaft_2014} and the references therein for more details on this subject. %\red{For the purpose of this work, a Dirac structure is the inverse graph of a skew-symmetric matrix.}

One difficulty of the framework is the non-uniqueness of the representation, and in particular, the choice of an appropriate Hamiltonian to avoid unnecessary computation, especially in this non-linear case. A first logical choice is to take the kinetic energy $E$ as Hamiltonian. However, it would lead to a modulated Dirac structure, \textit{i.e.}, a Dirac structure depending on the state variable (also known as energy variable, the variable of the Hamiltonian), or in other words, the non-linearity is taken \emph{inside} the Dirac structure. The proposed approach here is to consider a constant Dirac structure and to export the non-linearity inside the constitutive relation, which relates the state variable with the co-state variable (or co-energy variable, the derivative of $\mathcal{H}$ in the direction of the state variable).

More precisely, let us write
$
\partial_t v(t,x) = - \partial_x e(t,x)\;,
$
where $e := \delta_v \mathcal{H}$ is the co-state variable of a Hamiltonian $\mathcal{H}(v)$, and $\delta_v$ is the variational (or Fréchet) derivative of $\mathcal{H}$ in the direction $v$. Since we consider the Burgers' equation, it leads to $\delta_v \mathcal{H} = \frac{v^2}{2}$, \textit{i.e.}, to (up to an additive constant)
{\small
\begin{equation}
\label{eq:Hamiltonian}
\mathcal{H}(v(t,x)) := \int_0^1 \frac{v^3(t,x)}{6} \, \mathrm{d}x\;.
\end{equation}}

\begin{remark}
The Hamiltonian functional $\mathcal{H}$ is not a physical energy and is not sign-definite. It is introduced as a generating functional that allows the non-linear convective term of the BE to be entirely embedded in the constitutive relation, leading to a constant Dirac structure. As a consequence, passivity or dissipation w.r.t. $\mathcal{H}$ should not be interpreted in an energetic or thermodynamic sense.
\end{remark}

\begin{lemma}
\label{lem:power-balance-cube}
For smooth solutions, the Hamiltonian $\mathcal{H}$ satisfies the power balance
{\small
\begin{equation}
\label{eq:power-balance-smooth}
\frac{\mathrm{d}\mathcal{H}}{\mathrm{d}t} = \frac{v^4(t,0) - v^4(t,1)}{8}\;.
\end{equation}}
When a shock occurs at $x_s(t)$, it satisfies
{\small
\begin{equation}
\label{eq:power-balance-cube}
\frac{\mathrm{d}\mathcal{H}}{\mathrm{d}t} = \frac{v^4(t,0) - v^4(t,1)}{8} + \frac{(v_r - v_l)^3(v_r + v_l)}{24}\;.
\end{equation}}
\end{lemma}

\begin{proof}
The proof is similar to that of Lemma~\ref{lem:kinetic-energy}, once we observe that $\partial_t (\nicefrac{v^3}{6}) = -\partial_x (\nicefrac{v^4}{8})$.
\end{proof}

\vspace{-0.5em}
\subsection{Port-Hamiltonian representation}
\vspace{-0.5em}

Let us choose the velocity $v$ as state variable. Then, the co-state variable is $e := \delta_{v} \mathcal{H} = \frac{v^2}{2}$, as desired, since $\mathcal{H}$ has been designed for in that sense. Furthermore, we assume that the boundary control and observation are
{\small
$$
\left\lbrace
\begin{array}{rcl}
u_\ell(t) &=& \nicefrac{e(t,0)}{\sqrt{2}}\;, \\
u_r(t) &=& \nicefrac{e(t,1)}{\sqrt{2}}\;,
\end{array}
\right.
\quad
\left\lbrace
\begin{array}{rcl}
y_\ell(t) &=& \nicefrac{e(t,0)}{\sqrt{2}}\;, \\
y_r(t) &=& \nicefrac{-e(t,1)}{\sqrt{2}}\;.
\end{array}
\right.
$$}
Then, the power balance~\eqref{eq:power-balance-cube} reads
{\small
$$
\frac{\mathrm{d}\mathcal{H}}{\mathrm{d}t} = u_\ell(t) y_\ell(t) + u_r(t) y_r(t) + \frac{(v_r - v_l)^3(v_r + v_l)}{24}\;,
$$}
and the pH system reads
{\small
\begin{equation}
\label{eq:pH-system}
\left\lbrace
\begin{array}{rcl}
\partial_t v = -\partial_x e\;, &\quad& e = \nicefrac{v^2}{2}\;, \\
u_\ell = \nicefrac{\gamma_0 e}{\sqrt{2}}\;, &\quad&
y_\ell = \nicefrac{\gamma_0 e}{\sqrt{2}}\;, \\
u_r = \nicefrac{\gamma_1 e}{\sqrt{2}}\;, &\quad&
y_r = \nicefrac{-\gamma_1 e}{\sqrt{2}}\;,
\end{array}
\right.
\end{equation}}
where $\gamma_0$, $\gamma_1$, are the Dirichlet traces at $x=0$, $x=1$.

\begin{theorem}
\label{th:Dirac}
Let us denote the:
\begin{itemize}
\item structure operator: {\small$J := -\partial_x \in \mathcal{L}(H^1(0,1);L^2(0,1))$};
\item control operator: {\small$G := \frac{1}{\sqrt{2}} \begin{bmatrix} \gamma_0 \\ \gamma_1 \end{bmatrix} \in \mathcal{L}(H^1(0,1);\mathbb{R}^2)$};
\item observation operator: {\small$K := \frac{1}{\sqrt{2}} \begin{bmatrix} \gamma_0 \\ -\gamma_1 \end{bmatrix} \in \mathcal{L}(H^1(0,1);\mathbb{R}^2)$};
\end{itemize}
\begin{multline}
\label{eq:Dirac-structure}
\text{then,} \; \mathcal{D} := \Bigg\{ \left( \begin{pmatrix} f \\ f_\partial \end{pmatrix}, \begin{pmatrix} e \\ e_\partial \end{pmatrix} \right) \; \mid \; e \in H^1(0,1), \, f = J e, \\ f_\partial = - K e, \, e_\partial = G e \Bigg\} \;,
\end{multline}
is a Dirac structure on $\mathcal{B} := \left( L^2(0,1) \times \mathbb{R}^2 \right)^2$ endowed with the symmetrized bilinear pairing
{\small
\begin{multline*}
\left[ \begin{pmatrix} \bm{f} \\ \bm{e} \end{pmatrix}, \begin{pmatrix} \widetilde{\bm{f}} \\ \widetilde{\bm{e}} \end{pmatrix} \right]_{\mathcal{B}}
:= \int_0^1 f(x) \; \widetilde e(x) \, \mathrm{d}x 
+ \int_0^1 \widetilde f(x) \; e(x) \, \mathrm{d}x \\
+ f_\partial \cdot \widetilde e_\partial
+ \widetilde f_\partial \cdot e_\partial\;,
\end{multline*}}
for all
{\small
$
\begin{pmatrix} \bm{f} \\ \bm{e} \end{pmatrix} :=
\begin{pmatrix}
f &
f_\partial &
e &
e_\partial
\end{pmatrix}^\top,
\;
\begin{pmatrix} \widetilde{\bm{f}} \\ \widetilde{\bm{e}} \end{pmatrix} :=
\begin{pmatrix}
\widetilde f &
\widetilde f_\partial &
\widetilde e &
\widetilde e_\partial
\end{pmatrix}^\top\;
\in \mathcal{B}.
$}
\end{theorem}

\begin{proof}
A Dirac structure on $\mathcal{B}$ is a maximal isotropic subspace for this pairing, \textit{i.e.}, a subspace $\mathcal{D}$ is a Dirac structure on $\mathcal{B}$ if and only if
$
\mathcal{D}^{\perp_\mathcal{B}} = \mathcal{D}\;,
$
where $\mathcal{D}^{\perp_\mathcal{B}}$ is the orthogonal companion of $\mathcal{D}$ w.r.t. the pairing $\left[ \cdot, \cdot \right]_{\mathcal{B}}$\;.

Let $\begin{pmatrix} \bm{f} \\ \bm{e} \end{pmatrix} := \begin{pmatrix}
f &
f_\partial &
e &
e_\partial
\end{pmatrix}^\top$ be in $\mathcal{D}$ given by~\eqref{eq:Dirac-structure}. Then, for all $\begin{pmatrix} \widetilde{\bm{f}} \\ \widetilde{\bm{e}} \end{pmatrix} := \begin{pmatrix}
\widetilde f &
\widetilde f_\partial &
\widetilde e &
\widetilde e_\partial
\end{pmatrix}^\top \in \mathcal{D}$, one computes
$$
\begin{array}{rcl}
\left[ \begin{pmatrix} \bm{f} \\ \bm{e} \end{pmatrix}, \begin{pmatrix} \widetilde{\bm{f}} \\ \widetilde{\bm{e}} \end{pmatrix} \right]_{\mathcal{B}}
&=& - \int_0^1 \partial_x e(x) \; \widetilde e(x) \, \mathrm{d}x  \\
&\quad&
- \int_0^1 \partial_x \widetilde e(x) \; e(x) \, \mathrm{d}x
+ 2 \left[ \frac{e \, \widetilde e}{2} \right]_0^1\;, \\
&=& - \left[ e \, \widetilde e \right]_0^1 + \left[ e \, \widetilde e \right]_0^1 = 0\;.
\end{array}
$$
Hence, $\mathcal{D} \subset \mathcal{D}^{\perp_{\mathcal{B}}}$\;.

Reciprocally, let $\begin{pmatrix}
f &
f_\partial &
e &
e_\partial
\end{pmatrix}^\top \in \mathcal{D}^{\perp_{\mathcal{B}}}$, then, for all $\begin{pmatrix}
\widetilde f &
\widetilde f_\partial &
\widetilde e &
\widetilde e_\partial
\end{pmatrix}^\top \in \mathcal{D}$, one has
\begin{center}
$
0 = \int_0^1 f(x) \; \widetilde e(x) \, \mathrm{d}x 
+ \int_0^1 \widetilde f(x) \; e(x) \, \mathrm{d}x 
+ f_\partial \cdot G \widetilde e
- K \widetilde e \cdot e_\partial\;,
$
\end{center}
for all $\widetilde e \in H^1(0,1)$. Taking the $L^2$-basis $\widetilde e(x) = \sin(\kappa \pi x)$, $\kappa \in \mathbb{N}$ and $\widetilde e(x) = \cos(\kappa \pi x)$, $\kappa \in \mathbb{N}\cup\{0\}$, implies that $f = -\partial_x e = J e$ and the boundary terms
$
f_\partial = - K e, \, e_\partial = G e,
$
and then, $\mathcal{D}^{\perp_{\mathcal{B}}} \subset \mathcal{D}$\;, hence $\mathcal{D} = \mathcal{D}^{\perp_{\mathcal{B}}}$\;.
\end{proof}

\begin{remark}
The Dirac structure of Theorem~\ref{th:Dirac} does encode the power balance~\eqref{eq:power-balance-smooth} of the Hamiltonian $\mathcal{H}$ for smooth solutions. However, the dissipation related to a shock, as in~\eqref{eq:power-balance-cube}, is not contained in $\mathcal{D}$ as is.
\end{remark}

\begin{remark}
Note that~\cite[Theorem~4.3]{Kurula_2010} is not applicable since the range of the operator $\begin{bmatrix} G & K \end{bmatrix}^\top$ is not dense in $\mathbb{C}^2 \oplus \mathbb{C}^2$ (Hilbert spaces must be complex-valued in that framework).
\end{remark}

\vspace{-0.5em}
\subsection{Structure-preserving discretization}
\label{sec:discrete-inviscid}
\vspace{-0.5em}

Let us consider the pH system~\eqref{eq:pH-system} in its weak formulation, as follows: we are seeking for $v$ and $e$ such that for all test functions $\varphi \in H^1(0,1)$:
{\small
\begin{equation}
\label{eq:pH-system-weak}
\left\lbrace
\begin{array}{rcl}
\int_0^1 \partial_t v \, \varphi \, \mathrm{d}x &=& \int_0^1 e \, \partial_x \varphi \, \mathrm{d}x 
+ u_\ell \; \sqrt{2} \gamma_0 \varphi - u_r \; \sqrt{2} \gamma_1 \varphi \;, \\
\int_0^1 e \, \varphi \, \mathrm{d}x &=& \int_0^1 \frac{v^2}{2} \, \varphi \, \mathrm{d}x \;, \\
2 y_\ell &=& \sqrt{2} \gamma_0 e\;, \\
2 y_r &=& - \sqrt{2} \gamma_1 e\;,
\end{array}
\right.
\end{equation}}
where an integration by parts has been applied.

Let $\Phi := \{ \varphi^i \}_{i=0}^{N+1}$ that satisfy
$$
\varphi^0 = d_{0}\;,
\quad
\varphi^{N+1} = d_{1}\;,
$$
where $d$ is the Dirac mass, and $\{ \varphi^i \}_{i=1}^{N}$ be a finite element basis of $H^1_0(0,1)$, and define the vectors of coefficients
$$
\underline{v}(t) :=
\begin{pmatrix}
v_1(t) \\ \vdots \\ v_N(t)
\end{pmatrix} \in \mathbb{R}^N\;, \quad
\underline{e}(t) :=
\begin{pmatrix}
e_1(t) \\ \vdots \\ e_N(t)
\end{pmatrix} \in \mathbb{R}^N\;, 
$$
of the approximations $v_d := \sum_{i=1}^N v_i \; \varphi^i$ (resp. $e_d := \sum_{i=1}^N e_i \; \varphi^i$) of $v$ (resp. $e$) in the basis $\{ \varphi^i \}_{i=1}^{N}$.

Then, system~\eqref{eq:pH-system-weak} becomes
{\small
\begin{equation}
\label{eq:matrix}
\begin{bmatrix}
M & 0 & 0 \\
0 & 2 & 0 \\
0 & 0 & 2
\end{bmatrix}
\begin{pmatrix}
\nicefrac{\mathrm{d} \underline{v}(t)}{\mathrm{d}t} \\
- y_\ell(t) \\
- y_r(t)
\end{pmatrix}
=
\begin{bmatrix}
D & B_\ell & B_r \\
-B_\ell^\top & 0 & 0 \\
-B_r^\top & 0 & 0
\end{bmatrix}
\begin{pmatrix}
\underline{e}(t) \\
u_\ell(t) \\
u_r(t)
\end{pmatrix}\;,
\end{equation}}
together with the constitutive relation defined in weak form, \textit{i.e.}, for all $i = 1, \dots, N$
{\small
$$
\int_0^1 e_d(t,x) \, \varphi^i(x) \, \mathrm{d}x = \int_0^1 \frac{v_d^2(t,x)}{2} \, \varphi^i(x) \, \mathrm{d}x \;,
$$}
which rewrites using the coefficients, for all $i = 1, \dots, N$
\begin{center}
$
(M \underline{e}(t))_i = \left( \int_0^1 \varphi^i(x) \frac{\left( \sum_{j=1}^N \varphi^j(x) v_j(t) \right)^2}{2} \, \mathrm{d}x \right)_i \;,
$
\end{center}
where the matrices are given by
{\small
$$
\begin{array}{@{}rcl@{}}
\mathcal{M}_N(\mathbb{R}) \ni (M)_{ij} &:=& \int_0^1 \varphi^j(x) \, \varphi^i(x) \, \mathrm{d}x \;, \\
\mathcal{M}_N(\mathbb{R}) \ni (D)_{ij} &:=& \int_0^1 \varphi^j(x) \, \partial_x \varphi^i(x) \, \mathrm{d}x \;, \\
\mathbb{R}^N \ni (B_\ell)_i &:=& \sqrt{2} \delta_{0i} \;, \quad
\mathbb{R}^N \ni (B_r)_i := -\sqrt{2} \delta_{(N+1)i} \;,
\end{array}
$$}
where $\delta$ is the Kronecker symbol.

\begin{remark}
The constitutive relation is enforced in a Galerkin sense; in particular, $e_d$ is the $L^2$-projection of $\nicefrac{v_d^2}{2}$ onto the finite element space.
\end{remark}

\begin{proposition}
\label{prop:D-skew-sym}
$D$ is skew-symmetric, \textit{i.e.}, $D^\top = -D$\;.
\end{proposition}

\begin{proof}
Let us start by noticing that the diagonal entries are null. For all $\varphi^i$, either $i\in\{0,N+1\}$, and $(D)_{ii} = 0$ trivially, or $0<i<N+1$ and the support of $\varphi^i$ is inside $(0,1)$ and hence
\begin{center}
$
(D)_{ii} = \int_0^1 \varphi^i(x) \, \partial_x \varphi^i(x) \, \mathrm{d}x
= - \int_0^1 \partial_x \varphi^i(x) \, \varphi^i(x) \, \mathrm{d}x \;,
$
\end{center}
implying $(D)_{ii} = 0$.
Now, for $i\neq j$, it is clear that $\int_0^1 \varphi^i(x) \, \partial_x \varphi^j(x) \, \mathrm{d}x = - \int_0^1 \varphi^j(x) \, \partial_x \varphi^i(x) \, \mathrm{d}x$.
\end{proof}

\begin{remark}
When dealing with non-homogeneous Dirichlet boundary conditions, one discretization process consists in a finite element family for $H^1_0(0,1)$ and boundary value enforcement. This is precisely what is done here when choosing $\Phi$ as Kronecker symbols at the boundaries, and as a finite element family \emph{inside} the domain.
\end{remark}

\begin{proposition}
The discrete Hamiltonian
$
\mathcal{H}^d(\underline{v}(t)) := \mathcal{H}(v_d(t,x))\;,
$
satisfies the discrete power balance
{\small
$$
\frac{\mathrm{d}}{\mathrm{d}t} \mathcal{H}^d(\underline{v}(t)) = u_\ell(t) \, 2 \, y_\ell(t) + u_r(t) \, 2 \, y_r(t) \;,
$$}
along the smooth solutions, \textit{i.e.}, in $\mathcal{C}^1([0,\infty); H^1(0,1))$.
\end{proposition}

\begin{remark}
The factor $2$ appearing in the power balance must be understood as the metric change at the boundary induced by the structure-preserving discretization process. Indeed, it is mandatory to ensure that~\eqref{eq:matrix} defines a Dirac structure at the discrete level: the block diagonal matrix $\mathrm{Diag}(M, 2, 2)$ defines the scalar product to consider on $\mathbb{R}^{N+2}$ for the discretization. This is well-known and understood in higher dimension, where the boundary mass matrix appears naturally in the discrete power balance.
\end{remark}

\begin{proof}
Recall that $v_d = \sum_{i=1}^N \varphi^i(x) v_i(t)$ and that $e_d = \nicefrac{v_d^2}{2}$. Then
$$
\begin{array}{rcl}
&& \frac{\mathrm{d}}{\mathrm{d}t} \mathcal{H}^d(\underline{v}(t)) 
= \frac{\mathrm{d}}{\mathrm{d}t} \mathcal{H}(v_d(t,x)) \;, \\
&=& \int_0^1 \frac{v_d^2(t,x)}{2} \, \partial_t v_d(t,x) \, \mathrm{d}x 
= \int_0^1 e_d(t,x) \, \partial_t v_d(t,x) \, \mathrm{d}x \;, \\
&=& \int_0^1 e_d(t,x) \, \partial_x e_d(t,x) \, \mathrm{d}x 
+ u_\ell(t) \, 2 \, y_\ell(t) + u_r(t) \, 2 \, y_r(t) \;,
\end{array}
$$
where the last equality comes from~\eqref{eq:pH-system-weak} with $\varphi = e_d$.

Finally
$
\frac{\mathrm{d}}{\mathrm{d}t} \mathcal{H}^d(\underline{v}(t)) = u_\ell(t) \, 2 \, y_\ell(t) + u_r(t) \, 2 \, y_r(t) \;,
$
since
\begin{center}
$
\int_0^1 e_d(t,x) \, \partial_x e_d(t,x) \, \mathrm{d}x = - \int_0^1 \partial_x e_d(t,x) \, e_d(t,x) \, \mathrm{d}x \;,
$
\end{center}
from $e_d \in H^1_0(0,1)$, implying this integral equals $0$.
\end{proof}

\vspace{-0.5em}
\section{Viscous model}
\label{sec:viscous-model}
\vspace{-0.5em}

Following~\cite{Bonkile_2018}, the VBE reads
{\small
\begin{equation}
\label{eq:pH-viscous-PDE}
\partial_t v(t,x) + \partial_x \frac{v^2(t,x)}{2} = \nu \partial^2_{xx} v(t,x), \quad \forall t \ge 0, x \in (0,1),
\end{equation}}
where $\nu > 0$ is the viscosity parameter.

As can be observed in Section~\ref{sec:simulations}, shocks are not handled by the discrete model so far. A common approach to dealing with shocks is to consider numerical viscosity (see, \textit{e.g.},~\cite{LeVeque_1992}) which smooths the solutions.

More precisely, for any viscosity parameter $\nu>0$, the VBE is parabolic and thus regularizing. In particular, starting from sufficiently regular initial data, the solution remains smooth for all $t>0$, and no shock discontinuity can form. What would correspond to a shock in the inviscid limit is replaced by a smooth internal layer of thickness $O(\nu)$, across which the energy dissipation is entirely accounted for by the viscous term $\nu \int_0^1 (\partial_x v(t,x))^2 \, \mathrm{d}x$. In turn, in the viscous case, there is no separate ``shock dissipation'': the entropy dissipation of the inviscid limit is recovered as the vanishing-viscosity limit of the viscous dissipation.

This motivates the introduction of viscosity into the pH framework, not as an artificial numerical device, but as a physical regularization whose dissipation concentrates near shocks in the inviscid limit. In particular, dissipation here refers to physical energy dissipation, not to a monotonic decay of the Hamiltonian $\mathcal{H}$.

\vspace{-0.5em}
\subsection{Port-Hamiltonian representation}
\vspace{-0.5em}

\begin{lemma}
Along the solutions $v$ of~\eqref{eq:pH-viscous-PDE}, the same Hamiltonian $\mathcal{H}$ defined by~\eqref{eq:Hamiltonian} satisfies the power balance
{\small
\begin{multline*}
\frac{\mathrm{d} \mathcal{H}}{\mathrm{d}t} = 
\frac{v^4(t,0) - v^4(t,1)}{8} 
+ \nu \left[ \frac{v^2(t,x)}{2} \, \partial_x v(t,x) \right]_0^1 \\
- \nu \int_0^1 v(t,x) \, \left( \partial_x v(t,x) \right)^2 \, \mathrm{d}x \;,
\end{multline*}}
and the kinetic energy follows
{\small
$$
\frac{\mathrm{d} E}{\mathrm{d}t} = 
\frac{v^3(t,0) - v^3(t,1)}{3} 
+ \nu \left[ v(t,x) \, \partial_x v(t,x) \right]_0^1
- \nu \int_0^1 \left( \partial_x v(t,x) \right)^2 \, \mathrm{d}x.
$$}
\end{lemma}

\begin{proof}
Similar to Lemmas~\ref{lem:kinetic-energy} and~\ref{lem:power-balance-cube}.
\end{proof}

\begin{remark}
Assume that the velocity profile $v(t,\cdot)$ keeps a constant sign on $(0,1)$. Then the Hamiltonian $\mathcal{H}$ has the same sign, while the viscous term $-\nu \int_0^1 v (\partial_x v)^2$ has the opposite sign. This confirms its dissipative nature, as viscosity tends to drive $\mathcal{H}$ back toward zero.
\end{remark}

To consider the dissipation, let us add a resistive port $(f_r, e_r)$ and the associated boundary ports $(u_\ell^\nu,y_\ell^\nu)$ and $(u_r^\nu,y_r^\nu)$ to~\eqref{eq:pH-system}, as follows
{\small
\begin{equation}
\label{eq:pH-viscous-system}
\left\lbrace
\begin{array}{rcl}
\partial_t v = -\partial_x e - \partial_x e_r\;, &\quad& e = \nicefrac{v^2}{2}\;, \\
f_r = -\partial_x e &\quad& v \, e_r = \nu f_r \;, \\
u_\ell = \nicefrac{\gamma_0 e}{\sqrt{2}}\;, &\quad&
y_\ell = \nicefrac{\gamma_0 e}{\sqrt{2}}\;, \\
u_r = \nicefrac{\gamma_1 e}{\sqrt{2}}\;, &\quad&
y_r = \nicefrac{-\gamma_1 e}{\sqrt{2}}\;, \\
u^\nu_\ell = \gamma_0 e\;, &\quad&
y^\nu_\ell = \gamma_0 e_r\;, \\
u^\nu_r = \gamma_1 e\;, &\quad&
y^\nu_r = -\gamma_1 e_r\;.
\end{array}
\right.
\end{equation}}

\begin{theorem}
\label{th:Dirac-viscous}
Let $J$, $G$ and $K$ be as in Theorem~\ref{th:Dirac}. Let us denote furthermore the resistive:
\begin{itemize}
\item structure operator: {\small$R := \partial_x \in \mathcal{L}(H^1(0,1);L^2(0,1))$};
\item control operator: {\small$G^\nu := \begin{bmatrix} \gamma_0 \\ \gamma_1 \end{bmatrix} \in \mathcal{L}(H^1(0,1);\mathbb{R}^2)$};
\item observation operator: {\small$K^\nu := \begin{bmatrix} \gamma_0 \\ -\gamma_1 \end{bmatrix} \in \mathcal{L}(H^1(0,1);\mathbb{R}^2)$};
\end{itemize}
\begin{multline}
\label{eq:Dirac-structure-viscous}
\text{then,} \; \mathcal{D}^\nu := \Bigg\{ \left( \begin{pmatrix} f \\ f_r \\ f_\partial \\ f^\nu_\partial \end{pmatrix}, \begin{pmatrix} e \\ e_r \\ e_\partial \\ e^\nu_\partial \end{pmatrix} \right) \; \mid \; e \in H^1(0,1), \\
e_r \in H^1(0,1), \, f = J e - R e_r, \, f_r = - R e \\
f_\partial = - K e, \, e_\partial = G e, \, f^\nu_\partial = - K^\nu e_r, \, e^\nu_\partial = G^\nu e \Bigg\} \;,
\end{multline}
is a Dirac structure on $(L^2(0,1) \times L^2(0,1) \times \mathbb{R}^2 \times \mathbb{R}^2)^2$ endowed with the bilinear symmetrized product inherited from the scalar product on $L^2(0,1) \times L^2(0,1) \times \mathbb{R}^2 \times \mathbb{R}^2$.
\end{theorem}

\begin{proof}
The proof follows the same integration-by-parts arguments as Theorem~\ref{th:Dirac}. The additional distributed and boundary port variables generate boundary terms that cancel exactly by construction through the definitions of $G^\nu$ and $K^\nu$, yielding maximal isotropy of $\mathcal{D}^\nu$.
\end{proof}

\begin{remark}
Contrary to the Dirac structure $\mathcal{D}$ of Theorem~\ref{th:Dirac}, $\mathcal{D}^\nu$ defined by~\eqref{eq:Dirac-structure-viscous} does encode the power balance satisfied by the Hamiltonian, including the dissipative phenomena of a shock at the vanishing-viscosity limit.
\end{remark}

\vspace{-1em}
\subsection{Structure-preserving discretization}
\vspace{-0.5em}

Rewriting system~\eqref{eq:pH-viscous-system} in weak formulation and projecting on the basis defined in Section~\ref{sec:discrete-inviscid}, we obtain
{\small
\begin{equation}
\label{eq:VBE-discrete}
\bm{M}
\begin{pmatrix}
\nicefrac{\mathrm{d} \underline{v}(t)}{\mathrm{d}t} \\
\underline{f_r}(t) \\
- y_\ell(t) \\
- y_r(t) \\
- y^\nu_\ell(t) \\
- y^\nu_r(t) \\
\end{pmatrix} 
=
\begin{bmatrix}
D & -R & B_\ell & B_r & 0 & 0 \\
R^\top & 0 & 0 & 0 & B^\nu_\ell & B^\nu_r \\
-B_\ell^\top & 0 & 0 & 0 & 0 & 0 \\
-B_r^\top & 0 & 0 & 0 & 0 & 0 \\
0 & -(B^\nu_\ell)^\top & 0 & 0 & 0 & 0 \\
0 & -(B^\nu_r)^\top & 0 & 0 & 0 & 0 \\
\end{bmatrix}
\begin{pmatrix}
\underline{e}(t) \\
\underline{e_r}(t) \\
u_\ell(t) \\
u_r(t) \\
u^\nu_\ell(t) \\
u^\nu_r(t)
\end{pmatrix}
\end{equation}}
with $\bm{M} := \mathrm{Diag}(M, 2, 2, 1, 1)$,
together with the constitutive relations defined in weak form, \textit{i.e.}, for all $i = 1, \dots, N$
$
\int_0^1 e_d(t,x) \, \varphi^i(x) \, \mathrm{d}x = \int_0^1 \frac{v_d^2(t,x)}{2} \, \varphi^i(x) \, \mathrm{d}x \;,
$
and
{\small
\begin{equation}
\label{eq:CR-viscous}
\int_0^1 v_d(t,x) e_{rd}(t,x) \, \varphi^i(x) \, \mathrm{d}x = \int_0^1 \nu \, f_{rd}(t,x) \, \varphi^i(x) \, \mathrm{d}x \;,
\end{equation}}
where $e_{rd} := \sum_{i=1}^N e_{ri} \varphi^i$ and $f_{rd} := \sum_{i=1}^N f_{ri} \varphi^i$, the new matrices are given by
{\small
$$
\begin{array}{@{}rcl@{}}
\mathcal{M}_N(\mathbb{R}) \ni (R)_{ij} &:=& \int_0^1 \partial_x \varphi^j(x) \, \varphi^i(x) \, \mathrm{d}x \;, \\
\mathbb{R}^N \ni (B^\nu_\ell)_i &:=& -\delta_{0i} \;, \quad
\mathbb{R}^N \ni (B^\nu_r)_i := \delta_{(N+1)i} \;.
\end{array}
$$}

\begin{proposition}
The discrete Hamiltonian $\mathcal{H}^d$ satisfies the discrete power balance
{\small
\begin{multline*}
\frac{\mathrm{d}}{\mathrm{d}t} \mathcal{H}^d(\underline{v}(t)) = - \frac{1}{\nu} \, \int_\Omega v_d(t,x) \, \left( e_{rd}(t,x) \right)^2 \, \mathrm{d}x \\
+ u^\nu_\ell(t) y^\nu_\ell(t) + u^\nu_r(t) y^\nu_r(t) + u_\ell(t) \, 2 \, y_\ell(t) + u_r(t) \, 2 \, y_r(t) \;,
\end{multline*}}
along the solutions of~\eqref{eq:VBE-discrete}.
\end{proposition}

\begin{proof}
$$
\begin{array}{rcl}
 && \frac{\mathrm{d}}{\mathrm{d}t} \mathcal{H}^d(\underline{v}(t)) 
= \frac{\mathrm{d}}{\mathrm{d}t} \mathcal{H}(v_d(t,x)) , \\
&=& \int_0^1 \frac{v_d^2(t,x)}{2} \, \partial_t v_d(t,x) \, \mathrm{d}x 
= \int_0^1 e_d(t,x) \, \partial_t v_d(t,x) \, \mathrm{d}x , \\
&=& \int_0^1 \partial_x e_d(t,x) \, e_d(t,x) \, \mathrm{d}x  - \int_0^1 e_d(t,x) \, \partial_x e_{rd}(t,x) \, \mathrm{d}x \\
&& \qquad + u_\ell(t) \, 2 \, y_\ell(t) + u_r(t) \, 2 \, y_r(t) .
\end{array}
$$
As in the inviscid case, one has
$
\int_0^1 \partial_x e_d(t,x) \, e_d(t,x) \, \mathrm{d}x = 0 ,
$
and finally, one computes
$$
\begin{array}{rcl}
 && \int_0^1 e_{d}(t,x) \, \partial_x e_{rd}(t,x) \, \mathrm{d}x \\
&=& \int_0^1 f_{rd}(t,x) \, e_{rd}(t,x) \, \mathrm{d}x 
- u^\nu_\ell(t) y^\nu_\ell(t) - u^\nu_r(t) y^\nu_r(t) , \\
%&=& \frac{1}{\nu} \, \nu \, \int_0^1 f_{rd}(t,x) \, e_d(t,x) \, \mathrm{d}x 
%- u^\nu_\ell(t) y^\nu_\ell(t) - u^\nu_r(t) y^\nu_r(t) , \\
&=& \frac{1}{\nu} \, \int_0^1 v_d(t,x) \, (e_{rd}(t,x))^2 \, \mathrm{d}x 
- u^\nu_\ell(t) y^\nu_\ell(t) - u^\nu_r(t) y^\nu_r(t) ,
\end{array}
$$
where the first equality comes from $f_r = -\partial_x e$ in a variational sense, with test function $\varphi = e_{rd}$ and an integration by parts (second line of~\eqref{eq:VBE-discrete}), and the last one comes from the constitutive relation~\eqref{eq:CR-viscous} with $\varphi = e_{rd}$. Then, the result follows.
\end{proof}

\vspace{-0.5em}
\section{Numerical simulations}
\label{sec:simulations}
\vspace{-0.5em}

The simulations are carried out using SCRIMP package (Simulation and ContRol of Interactions in Multi-Physics), an ongoing project for the numerical simulation of infinite-dimensional pH systems, see~\cite{Ferraro_2024}.%. At its core, SCRIMP builds on two main libraries: GetFEM, an open-source finite element library, and PETSc, the Portable, Extensible Toolkit for Scientific Computation. Mesh generation is typically supported via GMSH, a tool for creating three-dimensional finite element meshes, while visualization of the results is facilitated by ParaView. For further details on the definition and numerical simulation of infinite-dimensional pH systems with SCRIMP, see~\cite{Ferraro_2024}.

As we aim to deal with (near) discontinuity at the discrete level, we choose as initial data: $v(0,x) \!=\! \exp\{-50(x\!-\!\nicefrac{1}{2})^2\}$, and as boundary control: $u_\ell=u_r=u^\nu_\ell=u^\nu_r=0$. The final time is $t_f = 0.4$s, long enough for a (near) shock to appear and small enough to avoid boundary interaction.

The spatial discretization is obtained using $\mathbb{P}^2$ Lagrange finite elements, with a uniform mesh size parameter $h$ from $5.10^{-4}$ to $10^{-2}$. The time discretization is performed using a Crank-Nicolson scheme with an adaptive time step $\textrm{d}t$ initialized at $\alpha h$, $\alpha = 0.5, 1$, or $2$. The viscosity is taken as $\nu = \nicefrac{\beta h}{\alpha}$, with $\beta = 0, 1, 2$, or $5$. Note that the inviscid case is treated when $\beta=0$.

For this section, let us choose $h=5.10^{-4}$, $\alpha=\beta=1$. In Fig.~\ref{fig:inviscid}, we see the velocity profiles of the inviscid case at two times, which show the instability due to the shock development (red vertical line). In Fig.~\ref{fig:viscous}, the profiles for the viscous case show the expected behavior, with a (near) shock that develops and slides to the right. In Fig.~\ref{fig:hamiltonian}, we can appreciate the structure-preserving property of the proposed scheme. Furthermore, the viscous dissipation (in blue) is close to the theoretical inviscid shock dissipation (in green), consistently with the vanishing-viscosity framework. Locally, the viscous layer remains smooth and confined to a small number of elements proportional to $\nu$.

Although the power balance of the kinetic energy at the discrete level is not proved, the same nice-looking behavior shows in Fig.~\ref{fig:energy}, which depicts the physical kinetic energy.

\begin{figure}
\centering
\includegraphics[width=0.49\linewidth]{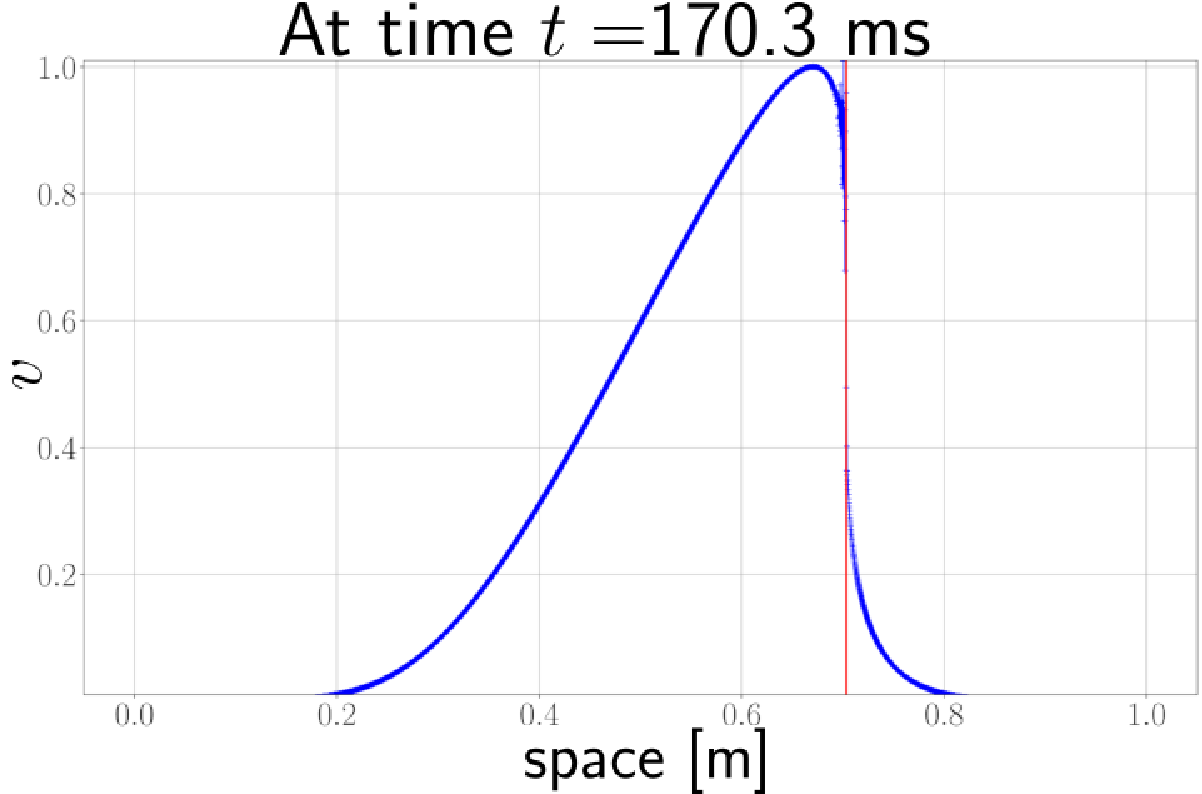}
\includegraphics[width=0.49\linewidth]{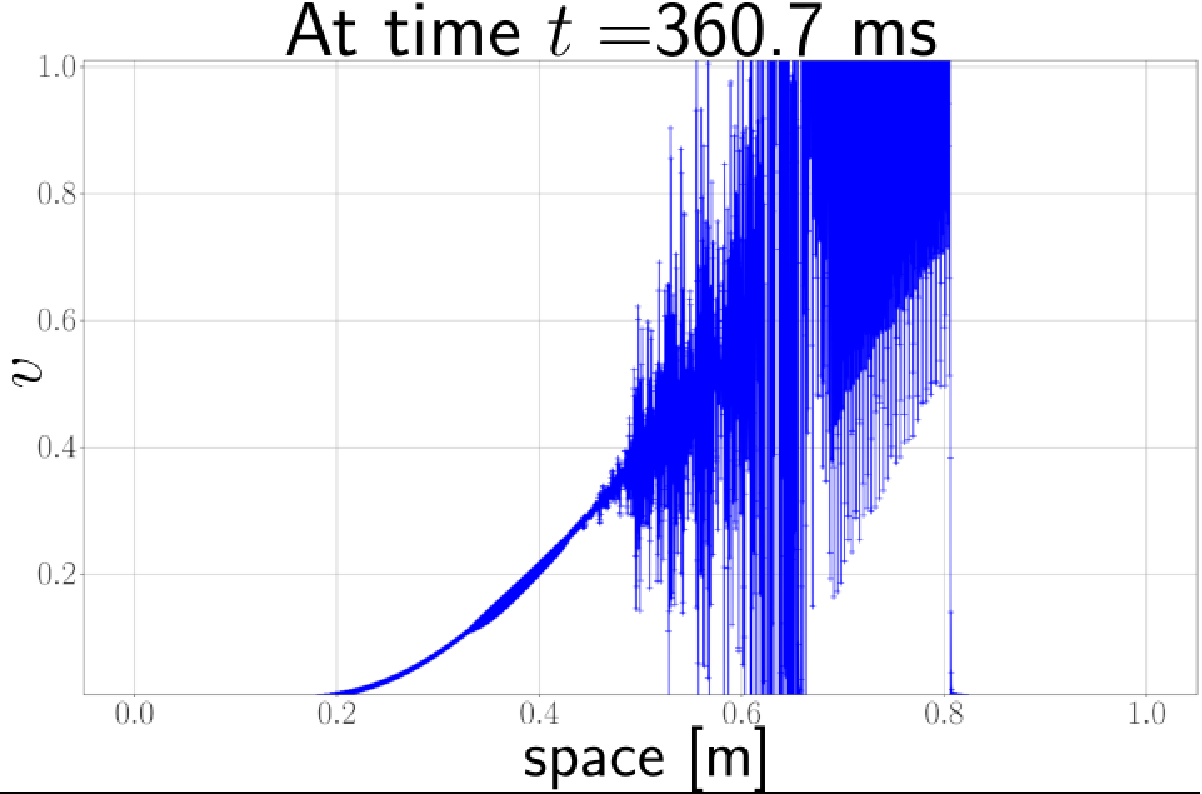}
\caption{\small Velocity profiles of the inviscid case}
\label{fig:inviscid}
\end{figure}

\begin{figure}
\centering
\includegraphics[width=0.49\linewidth]{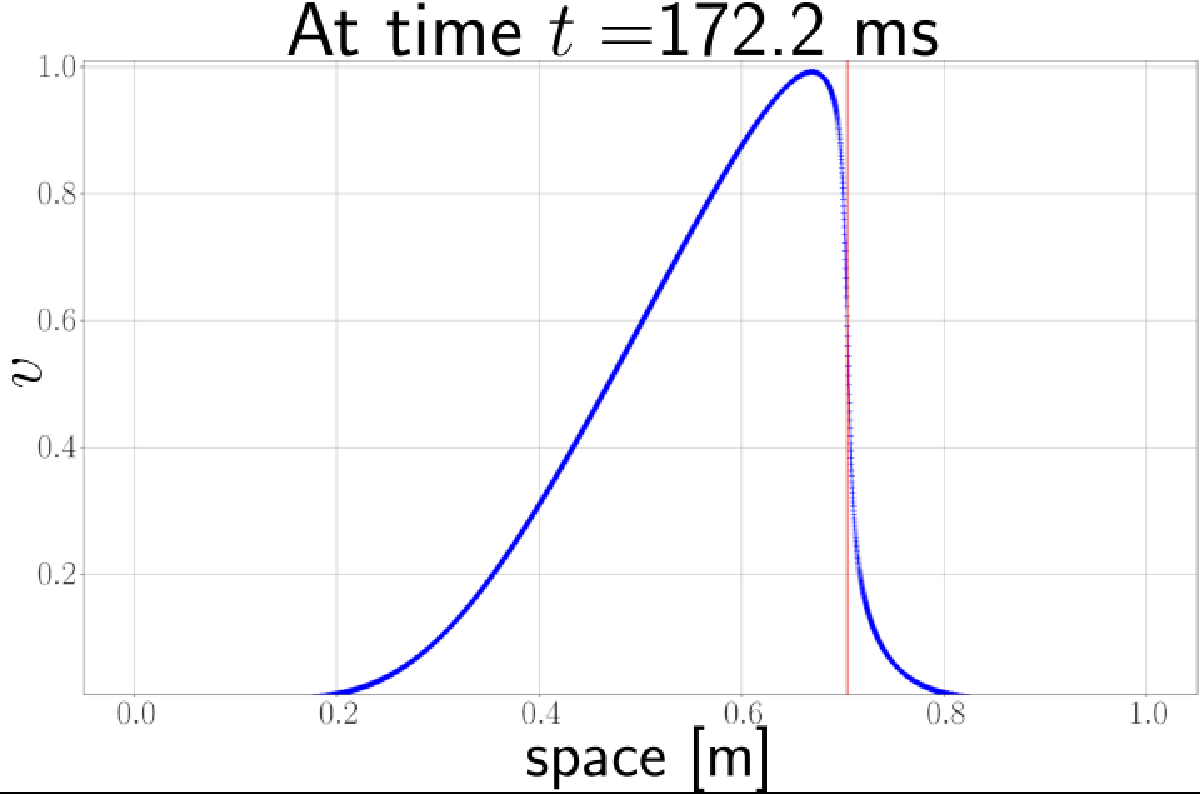}
\includegraphics[width=0.49\linewidth]{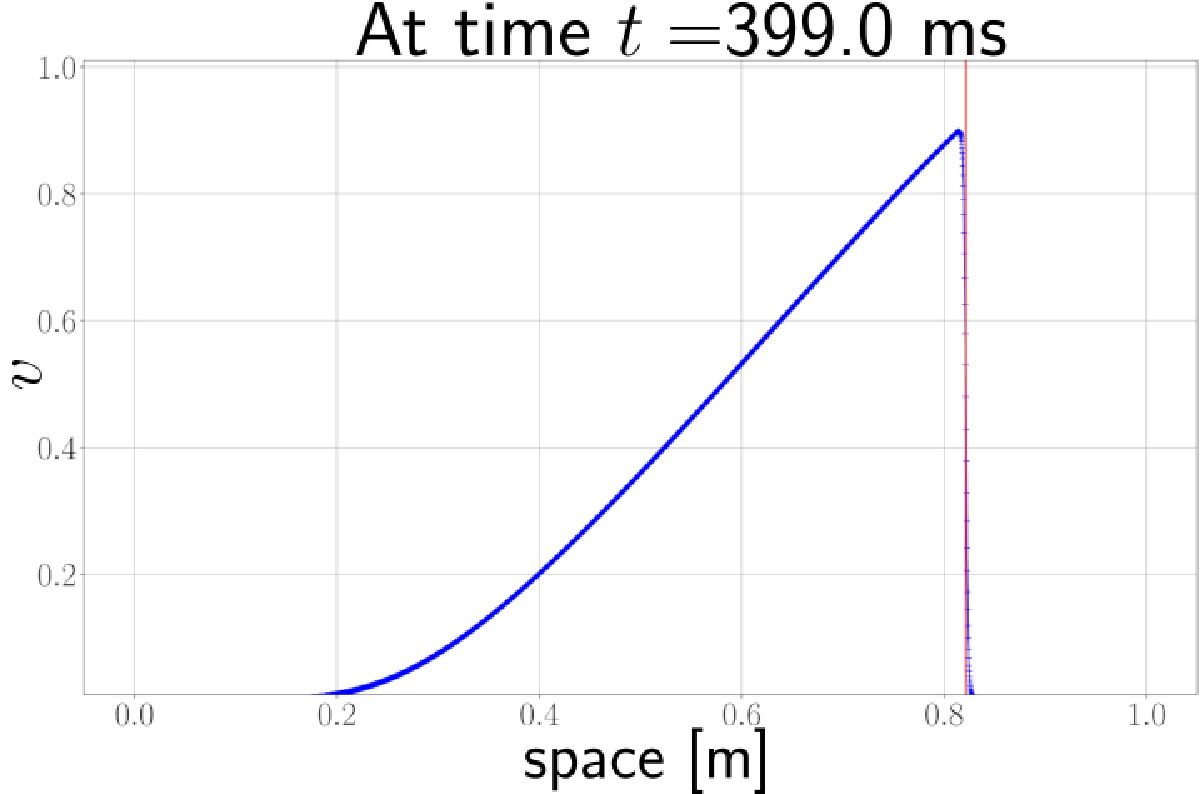}
\caption{\small Velocity profiles of the viscous case}
\label{fig:viscous}
\end{figure}

\begin{figure}
\centering
\includegraphics[width=0.9\linewidth]{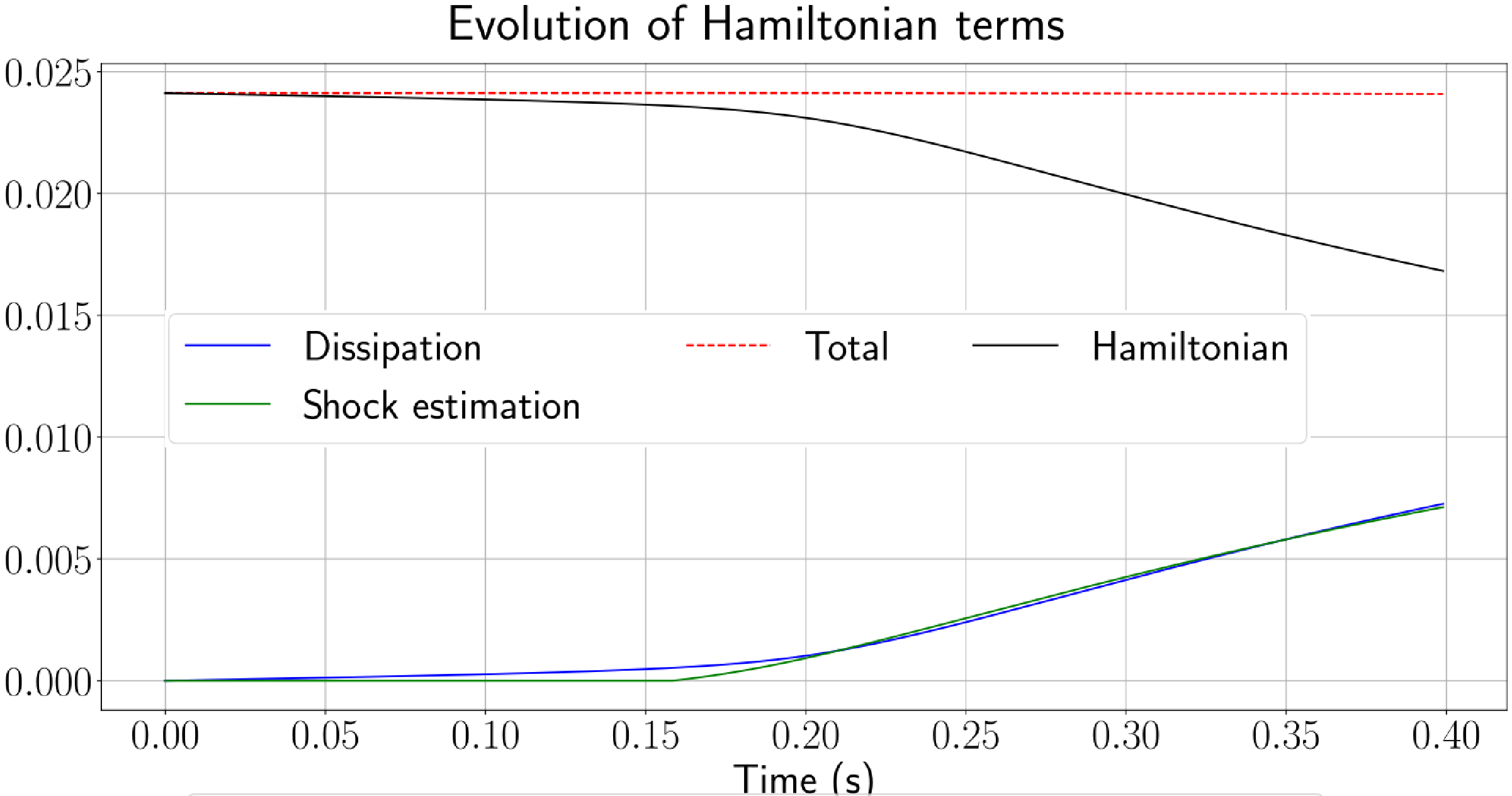}
\caption{\small Hamiltonian and dissipation evolution of the viscous case}
\label{fig:hamiltonian}
\end{figure}

\begin{figure}
\centering
\includegraphics[width=0.9\linewidth]{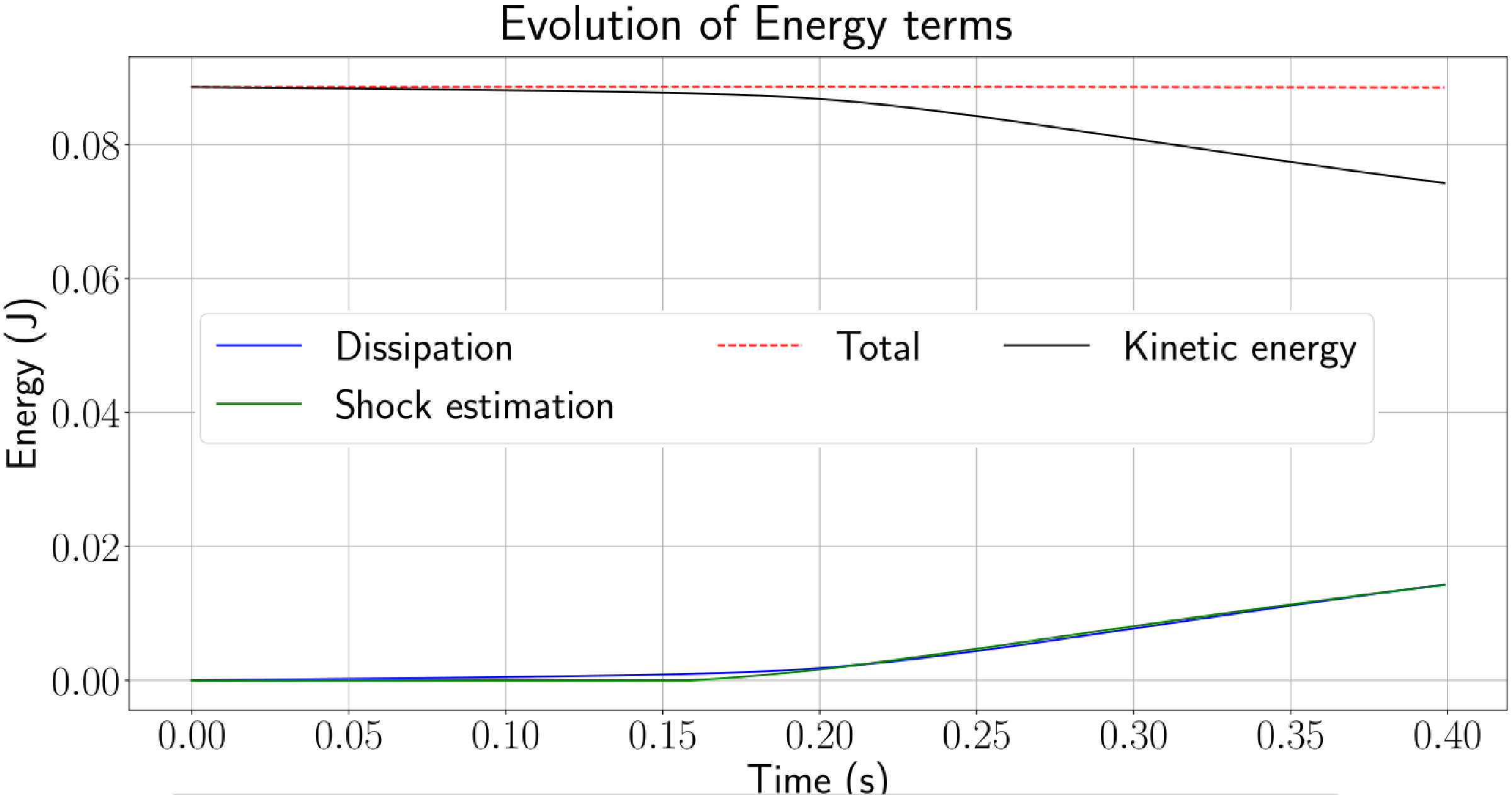}
\caption{\small Kinetic energy and dissipation evolution of the viscous case}
\label{fig:energy}
\end{figure}

The purpose of the numerical experiments is not to outperform classical stabilized schemes, but to illustrate that the proposed structure-preserving discretization maintains consistent power balances at the discrete level, even in regimes where steep gradients develop.

\vspace{-0.5em}
\section{Stability}
\label{sec:stability}
\vspace{-0.5em}

This section does not aim at establishing a theoretical stability result but provides empirical evidence that preserving discrete power balances is a key ingredient for numerical robustness.

We recall that a time adaptive scheme has been used, and some experiments reached the minimal allowed time step before $t\simeq0.4$. The following Tables~\ref{tab:0.5}--\ref{tab:1}--\ref{tab:2} summarize the energetic stability indicators for all tested configurations: the above number is the maximal relative variation $Var$, the below number is the final computed time $t_f$, and the number of time steps $Nb$ is in parentheses, as
{\small
\begin{tabular}{|c|}
\hline
$Var$ \\
$t_f$ ($Nb$) \\
\hline
\end{tabular}}\;.

\setlength{\tabcolsep}{1pt}
\renewcommand{\arraystretch}{2.}
\begin{table}[!ht]
\centering
\begin{tabular}{ |c||c|c|c|c|c| }
\hline 
\diagbox[width=1em,height=1.75em]{\;$\smash[t]{h}$}{$\smash[b]{\beta}$\;} & 5.0e-04 & 1.0e-03 & 2.5e-03 & 5.0e-03 & 1.0e-02 \\
\hline
\hline
0.0 & $\atop{5.141e-02}{0.355 (257)}$ & $\atop{2.874e-02}{0.400 (300)}$ & $\atop{2.891e-02}{0.400 (297)}$ & $\atop{2.895e-02}{0.399 (289)}$ & $\atop{6.621e-02}{0.399 (283)}$ \\
\hline
1.0 & $\atop{1.322e-03}{0.399 (247)}$ & $\atop{8.993e-04}{0.399 (160)}$ & $\atop{4.662e-04}{0.397 (96)}$ & $\atop{2.185e-04}{0.394 (69)}$ & $\atop{1.615e-05}{0.399 (107)}$ \\
\hline
2.0 & $\atop{8.957e-04}{0.398 (159)}$ & $\atop{5.581e-04}{0.397 (107)}$ & $\atop{2.193e-04}{0.397 (70)}$ & $\atop{4.616e-05}{0.399 (55)}$ & $\atop{1.911e-04}{0.017 (6)}$ \\
\hline
5.0 & $\atop{4.653e-04}{0.397 (95)}$ & $\atop{2.171e-04}{0.394 (69)}$ & $\atop{1.937e-05}{0.400 (53)}$ & $\atop{5.957e-05}{0.400 (90)}$ & $\atop{5.589e-03}{0.399 (243)}$ \\
\hline
\end{tabular}
\caption{\small Maximal relative variation of the Hamiltonian + dissipation (above number), and the final time and number of time steps (below numbers) for $\alpha = 0.5$}
\label{tab:0.5}

\begin{tabular}{ |c||c|c|c|c|c| }
\hline 
\diagbox[width=1em,height=1.75em]{\;$\smash[t]{h}$}{$\smash[b]{\beta}$\;} & 5.0e-04 & 1.0e-03 & 2.5e-03 & 5.0e-03 & 1.0e-02 \\
\hline
\hline
0.0 & $\atop{2.630e-02}{0.361 (263)}$ & $\atop{2.801e-02}{0.400 (301)}$ & $\atop{2.884e-02}{0.399 (295)}$ & $\atop{2.895e-02}{0.399 (288)}$ & $\atop{6.454e-02}{0.399 (282)}$ \\
\hline
1.0 & $\atop{1.526e-03}{0.399 (232)}$ & $\atop{1.321e-03}{0.399 (248)}$ & $\atop{7.822e-04}{0.399 (140)}$ & $\atop{4.701e-04}{0.400 (95)}$ & $\atop{2.065e-04}{0.395 (67)}$ \\
\hline
2.0 & $\atop{1.325e-03}{0.400 (248)}$ & $\atop{9.006e-04}{0.400 (161)}$ & $\atop{4.692e-04}{0.399 (96)}$ & $\atop{2.171e-04}{0.395 (68)}$ & $\atop{4.141e-05}{0.400 (140)}$ \\
\hline
5.0 & $\atop{7.815e-04}{0.399 (140)}$ & $\atop{4.648e-04}{0.396 (96)}$ & $\atop{1.528e-04}{0.391 (63)}$ & $\atop{2.336e-05}{0.391 (51)}$ & $\atop{1.461e-04}{0.400 (199)}$ \\
\hline
\end{tabular}
\caption{\small Maximal relative variation of the Hamiltonian + dissipation (above number), and the final time and number of time steps (below numbers) for $\alpha = 1$}
\label{tab:1}

\begin{tabular}{ |c||c|c|c|c|c| }
\hline 
\diagbox[width=1em,height=1.75em]{\;$\smash[t]{h}$}{$\smash[b]{\beta}$\;} & 5.0e-04 & 1.0e-03 & 2.5e-03 & 5.0e-03 & 1.0e-02 \\
\hline
\hline
0.0 & $\atop{2.682e-02}{0.351 (254)}$ & $\atop{2.741e-02}{0.400 (300)}$ & $\atop{2.884e-02}{0.399 (294)}$ & $\atop{2.956e-02}{0.399 (287)}$ & $\atop{6.823e-02}{0.399 (279)}$ \\
\hline
1.0 & $\atop{1.022e-02}{0.400 (260)}$ & $\atop{1.531e-03}{0.400 (233)}$ & $\atop{1.177e-03}{0.400 (214)}$ & $\atop{7.824e-04}{0.399 (138)}$ & $\atop{4.549e-04}{0.397 (91)}$ \\
\hline
2.0 & $\atop{1.525e-03}{0.399 (233)}$ & $\atop{1.323e-03}{0.399 (248)}$ & $\atop{7.824e-04}{0.399 (139)}$ & $\atop{4.663e-04}{0.400 (94)}$ & $\atop{1.398e-04}{0.395 (65)}$ \\
\hline
5.0 & $\atop{1.174e-03}{0.399 (215)}$ & $\atop{7.801e-04}{0.398 (140)}$ & $\atop{3.815e-04}{0.400 (85)}$ & $\atop{1.374e-04}{0.393 (61)}$ & $\atop{2.827e-04}{0.399 (77)}$ \\
\hline
\end{tabular}
\caption{\small Maximal relative variation of the Hamiltonian + dissipation (above number), and the final time and number of time steps (below numbers) for $\alpha = 2$}
\label{tab:2}
\end{table}

\vspace{-1em}
Across Tables~\ref{tab:0.5}--\ref{tab:1}--\ref{tab:2}, the inviscid configuration ($\beta=0$) exhibits comparatively large maximal relative variations (typically a few $10^{-2}$, up to $\approx 6\text{–}7\times 10^{-2}$), which indicates that the simulation becomes energetically unreliable once steep gradients develop. Introducing viscosity ($\beta\ge 1$) reduces this variation by one to several orders of magnitude in most cases (often $10^{-3}$, $10^{-4}$, and sometimes down to $10^{-5}$), confirming the stabilizing role of the dissipative extension. Importantly, the dependence on $h$ is consistent with the chosen scaling $\nu=\beta \nicefrac{h}{\alpha}$: for fixed $(\alpha,\beta)$, larger $h$ implies larger $\nu$ and generally smaller relative variations, whereas the finest meshes (small $h$, hence smaller $\nu$) can display larger variations (\textit{e.g.}, $\alpha=2,\beta=1,h=5\times10^{-4}$). Finally, a few parameter combinations lead to premature termination (very small final time) when the adaptive time step reaches its minimum, highlighting a practical stability boundary in $(\alpha,\beta,h)$ rather than a monotone trend with refinement. Nevertheless, it seems that larger values of $\alpha$ tend to reduce the number of time steps but may require stronger viscosity to maintain stability.

\FloatBarrier

\vspace{-0.5em}
\section{Conclusion}
\label{sec:conclusion}
\vspace{-0.5em}

We proposed a port-Hamiltonian formulation of the Burgers' equation based on a non-standard Hamiltonian, which allows the non-linear convective term to be embedded entirely in the constitutive relation while preserving a constant Dirac structure. This choice leads to a clear separation between the geometric interconnection structure and the non-linear dynamics, both at the continuous and discrete levels. A structure-preserving finite element discretization was then derived, relying on a weak enforcement of the constitutive relation, which ensures exact discrete power balances. The framework was further extended to the viscous Burgers' equation by introducing a resistive port and the associated boundary ports. This extension preserves the port-Hamiltonian structure while correctly accounting for viscous effects in the power balance.

Numerical simulations confirm the theoretical developments. In the viscous case, the dissipation is entirely governed by the viscous term, and the numerical results illustrate the stabilizing role of the proposed viscous regularization, while giving rise to a dissipation that is close to the decay of the inviscid limit. A systematic parametric study provides empirical evidence that the preservation of discrete power balances plays a central role in the numerical robustness of the method, even in regimes where classical discretizations are prone to instabilities.

Future work will focus on extending the proposed approach to higher-dimensional settings and on a deeper comparison with standard stabilized schemes, in particular for systems involving shocks or sharp internal layers.

%A conclusion section is not required. Although a conclusion may review
%the main points of the paper, do not replicate the abstract as the
%conclusion. A conclusion might elaborate on the importance of the work
%or suggest applications and extensions.

%\begin{ack}
%Place acknowledgments here.
%\end{ack}

%{\tiny\vspace{-0.25cm}
%\section*{DECLARATION OF GENERATIVE AI AND AI-ASSISTED TECHNOLOGIES IN THE WRITING PROCESS}\vspace{-0.5cm}
%The authors used LanguageTool to enhance the English wording and grammar. After using this tool, they reviewed and edited the content and took full responsibility for the content of the publication.
%\vspace{-0.25cm}}

{\small
\bibliographystyle{plain}
\bibliography{biblio}

@article{Bonkile_2018,
  title = {{A systematic literature review of Burgers' equation with recent advances}},
  volume = {90},
  ISSN = {0973-7111},
  DOI = {10.1007/s12043-018-1559-4},
  number = {6},
  journal = {Pramana},
  publisher = {Springer Science and Business Media LLC},
  author = {Bonkile, Mayur P and Awasthi, Ashish and Lakshmi, C and Mukundan, Vijitha and Aswin, V S},
  year = {2018},
}

@inbook{Burgers_1995,
  title = {{Mathematical Examples Illustrating Relations Occurring in the Theory of Turbulent Fluid Motion}},
  ISBN = {9789401101950},
  DOI = {10.1007/978-94-011-0195-0\_10},
  booktitle = {Selected Papers of J. M. Burgers},
  publisher = {Springer Netherlands},
  author = {Burgers, J. M.},
  year = {1995},
  pages = {281--334}
}

@book{Enflo_2004,
  author = {Bengt O. Enflo and Claes M. Hedberg},
  title = {{Theory of Nonlinear Acoustics in Fluids}},
  ISBN = {1402005725},
  DOI = {10.1007/0-306-48419-6},
  journal = {Fluid Mechanics and Its Applications},
  publisher = {Kluwer Academic Publishers},
  year = {2004}
}

@article{Ferraro_2024,
  title = {{Simulation and control of interactions in multi-physics, a Python package for port-Hamiltonian systems}},
  volume = {58},
  ISSN = {2405-8963},
  DOI = {10.1016/j.ifacol.2024.08.267},
  number = {6},
  journal = {IFAC-PapersOnLine},
  publisher = {Elsevier BV},
  author = {Ferraro, Giuseppe and Fournié, Michel and Haine, Ghislain},
  year = {2024},
  pages = {119--124}
}

@article{Kurula_2010,
  title = {{Dirac structures and their composition on Hilbert spaces}},
  volume = {372},
  ISSN = {0022-247X},
  DOI = {10.1016/j.jmaa.2010.07.004},
  number = {2},
  journal = {Journal of Mathematical Analysis and Applications},
  publisher = {Elsevier BV},
  author = {Kurula, Mikael and Zwart, Hans and van der Schaft, Arjan and Behrndt, Jussi},
  year = {2010},
  pages = {402--422}
}

@book{LeVeque_1992,
  title = {{Numerical Methods for Conservation Laws}},
  ISBN = {9783034886291},
  DOI = {10.1007/978-3-0348-8629-1},
  publisher = {Birkhäuser Basel},
  author = {LeVeque, Randall J.},
  year = {1992}
}

@article{Rashad_2020,
  title = {{Twenty years of distributed port-Hamiltonian systems: a literature review}},
  volume = {37},
  ISSN = {1471-6887},
  DOI = {10.1093/imamci/dnaa018},
  number = {4},
  journal = {IMA Journal of Mathematical Control and Information},
  publisher = {Oxford University Press (OUP)},
  author = {Rashad, Ramy and Califano, Federico and van der Schaft, Arjan J and Stramigioli, Stefano},
  year = {2020},
  pages = {1400--1422}
}

@article{van_der_Schaft_2014,
  title = {{Port-Hamiltonian Systems Theory: An Introductory Overview}},
  volume = {1},
  ISSN = {2325-6826},
  DOI = {10.1561/2600000002},
  number = {2–3},
  journal = {Foundations and Trends\textsuperscript{\tiny\textregistered} in Systems and Control},
  publisher = {Emerald},
  author = {van der Schaft, Arjan and Jeltsema, Dimitri},
  year = {2014},
  pages = {173--378}
}
}

%\appendix
%\section{}

\end{document}